\documentclass[reqno]{amsart}
\usepackage{amssymb}
\usepackage{eucal}
\usepackage{amsmath}
\usepackage{amscd}
\usepackage[dvips]{color}
\usepackage{multicol}
\usepackage[all]{xy}           
\usepackage{graphicx}
\usepackage{color}
\usepackage{colordvi}
\usepackage{xspace}
\usepackage{bookmark}

\usepackage[active]{srcltx} 

\usepackage{hyperref}
\usepackage{lipsum}

\newcommand{\nc}{\newcommand}
\newcommand{\delete}[1]{}
\nc{\mfootnote}[1]{\footnote{#1}} 
\nc{\todo}[1]{\tred{To do:} #1}

\delete{
\nc{\mlabel}[1]{\label{#1}}  
\nc{\mcite}[1]{\cite{#1}}  
\nc{\mref}[1]{\ref{#1}}  
\nc{\mbibitem}[1]{\bibitem{#1}} 
}

\nc{\mlabel}[1]{\label{#1}  
{\hfill \hspace{1cm}{\bf{{\ }\hfill(#1)}}}}
\nc{\mcite}[1]{\cite{#1}{{\bf{{\ }(#1)}}}}  
\nc{\mref}[1]{\ref{#1}{{\bf{{\ }(#1)}}}}  
\nc{\mbibitem}[1]{\bibitem[\bf #1]{#1}} 

\newtheorem{theorem}{Theorem}[section]
\newtheorem{definition}[theorem]{Definition}
\newtheorem{lemma}[theorem]{Lemma}
\newtheorem{corollary}[theorem]{Corollary}
\newtheorem{prop-def}[theorem]{Proposition-Definition}

\newtheorem{example}[theorem]{Example}

\newtheorem{proposition}[theorem]{Proposition}

\nc{\tred}[1]{\textcolor{red}{#1}}
\nc{\tblue}[1]{\textcolor{blue}{#1}}
\nc{\tgreen}[1]{\textcolor{green}{#1}}
\nc{\tpurple}[1]{\textcolor{purple}{#1}}
\nc{\btred}[1]{\textcolor{red}{\bf #1}}
\nc{\btblue}[1]{\textcolor{blue}{\bf #1}}
\nc{\btgreen}[1]{\textcolor{green}{\bf #1}}
\nc{\btpurple}[1]{\textcolor{purple}{\bf #1}}

\nc{\li}[1]{\textcolor{red}{Xiaomin:#1}}
\nc{\cm}[1]{\textcolor{blue}{Chengming: #1}}

\nc{\twovec}[2]{\left(\begin{array}{c} #1 \\ #2\end{array} \right )}
\nc{\threevec}[3]{\left(\begin{array}{c} #1 \\ #2 \\ #3 \end{array}\right )}
\nc{\twomatrix}[4]{\left(\begin{array}{cc} #1 & #2\\ #3 & #4 \end{array} \right)}
\nc{\threematrix}[9]{{\left(\begin{matrix} #1 & #2 & #3\\ #4 & #5 & #6 \\ #7 & #8 & #9 \end{matrix} \right)}}
\nc{\twodet}[4]{\left|\begin{array}{cc} #1 & #2\\ #3 & #4 \end{array} \right|}

\nc{\rk}{\mathrm{r}}

\nc{\gensp}{V} 
\nc{\relsp}{\Lambda} 
\nc{\leafsp}{X}    
\nc{\treesp}{\overline{\calt}} 

\nc{\vin}{{\mathrm Vin}}    
\nc{\lin}{{\mathrm Lin}}    

\nc{\gop}{{\,\omega\,}}     
\nc{\gopb}{{\,\nu\,}}
\nc{\svec}[2]{{\tiny\left(\begin{matrix}#1\\
#2\end{matrix}\right)\,}}  
\nc{\ssvec}[2]{{\tiny\left(\begin{matrix}#1\\
#2\end{matrix}\right)\,}} 

\nc{\su}{\mathrm{Su}}
\nc{\tsu}{\mathrm{TSu}}
\nc{\TSu}{\mathrm{TSu}}
\nc{\eval}[1]{{#1}_{\big|D}}
\nc{\oto}{\leftrightarrow}

\nc{\oaset}{\mathbf{O}^{\rm alg}}
\nc{\omset}{\mathbf{O}^{\rm mod}}
\nc{\oamap}{\Phi^{\rm alg}}
\nc{\ommap}{\Phi^{\rm mod}}
\nc{\ioaset}{\mathbf{IO}^{\rm alg}}
\nc{\iomset}{\mathbf{IO}^{\rm mod}}
\nc{\ioamap}{\Psi^{\rm alg}}
\nc{\iommap}{\Psi^{\rm mod}}

\nc{\suc}{{successor}\xspace} \nc{\Suc}{{Successor}\xspace}
\nc{\sucs}{{successors}\xspace} \nc{\Sucs}{{Successors}\xspace}
\nc{\Tsuc}{{T-successor}\xspace}
\nc{\Tsucs}{{T-successors}\xspace} \nc{\Lsuc}{{L-successor}\xspace}
\nc{\Lsucs}{{L-successors}\xspace} \nc{\Rsuc}{{R-successor}\xspace}
\nc{\Rsucs}{{R-successors}\xspace}

\nc{\bia}{{$\mathcal{P}$-bimodule ${\bf k}$-algebra}\xspace}
\nc{\bias}{{$\mathcal{P}$-bimodule ${\bf k}$-algebras}\xspace}

\nc{\rmi}{{\mathrm{I}}}
\nc{\rmii}{{\mathrm{II}}}
\nc{\rmiii}{{\mathrm{III}}}

\nc{\pll}{\beta}
\nc{\plc}{\epsilon}

\nc{\ass}{{\mathit{Ass}}}
\nc{\lie}{{\mathit{Lie}}}
\nc{\comm}{{\mathit{Comm}}}
\nc{\dend}{{\mathit{Dend}}}
\nc{\zinb}{{\mathit{Zinb}}}
\nc{\tdend}{{\mathit{TDend}}}
\nc{\prelie}{{\mathit{preLie}}}
\nc{\postlie}{{\mathit{PostLie}}}
\nc{\quado}{{\mathit{Quad}}}
\nc{\octo}{{\mathit{Octo}}}
\nc{\ldend}{{\mathit{ldend}}}
\nc{\lquad}{{\mathit{LQuad}}}

 \nc{\adec}{\check{;}} \nc{\aop}{\alpha}
\nc{\dftimes}{\widetilde{\otimes}} \nc{\dfl}{\succ} \nc{\dfr}{\prec}
\nc{\dfc}{\circ} \nc{\dfb}{\bullet} \nc{\dft}{\star}
\nc{\dfcf}{{\mathbf k}} \nc{\apr}{\ast} \nc{\spr}{\cdot}
\nc{\twopr}{\circ} \nc{\tspr}{\star} \nc{\sempr}{\ast}
\nc{\disp}[1]{\displaystyle{#1}}
\nc{\bin}[2]{ (_{\stackrel{\scs{#1}}{\scs{#2}}})}  
\nc{\binc}[2]{ \left (\!\! \begin{array}{c} \scs{#1}\\
    \scs{#2} \end{array}\!\! \right )}  
\nc{\bincc}[2]{  \left ( {\scs{#1} \atop
    \vspace{-.5cm}\scs{#2}} \right )}  
\nc{\sarray}[2]{\begin{array}{c}#1 \vspace{.1cm}\\ \hline
    \vspace{-.35cm} \\ #2 \end{array}}
\nc{\bs}{\bar{S}} \nc{\dcup}{\stackrel{\bullet}{\cup}}
\nc{\dbigcup}{\stackrel{\bullet}{\bigcup}} \nc{\etree}{\big |}
\nc{\la}{\longrightarrow} \nc{\fe}{\'{e}} \nc{\rar}{\rightarrow}
\nc{\dar}{\downarrow} \nc{\dap}[1]{\downarrow
\rlap{$\scriptstyle{#1}$}} \nc{\uap}[1]{\uparrow
\rlap{$\scriptstyle{#1}$}} \nc{\defeq}{\stackrel{\rm def}{=}}
\nc{\dis}[1]{\displaystyle{#1}} \nc{\dotcup}{\,
\displaystyle{\bigcup^\bullet}\ } \nc{\sdotcup}{\tiny{
\displaystyle{\bigcup^\bullet}\ }} \nc{\hcm}{\ \hat{,}\ }
\nc{\hcirc}{\hat{\circ}} \nc{\hts}{\hat{\shpr}}
\nc{\lts}{\stackrel{\leftarrow}{\shpr}}
\nc{\rts}{\stackrel{\rightarrow}{\shpr}} \nc{\lleft}{[}
\nc{\lright}{]} \nc{\uni}[1]{\tilde{#1}} \nc{\wor}[1]{\check{#1}}
\nc{\free}[1]{\bar{#1}} \nc{\den}[1]{\check{#1}} \nc{\lrpa}{\wr}
\nc{\curlyl}{\left \{ \begin{array}{c} {} \\ {} \end{array}
    \right .  \!\!\!\!\!\!\!}
\nc{\curlyr}{ \!\!\!\!\!\!\!
    \left . \begin{array}{c} {} \\ {} \end{array}
    \right \} }
\nc{\leaf}{\ell}       
\nc{\longmid}{\left | \begin{array}{c} {} \\ {} \end{array}
    \right . \!\!\!\!\!\!\!}
\nc{\ot}{\otimes} \nc{\sot}{{\scriptstyle{\ot}}}
\nc{\otm}{\overline{\ot}}
\nc{\ora}[1]{\stackrel{#1}{\rar}}
\nc{\ola}[1]{\stackrel{#1}{\la}}
\nc{\pltree}{\calt^\pl}
\nc{\epltree}{\calt^{\pl,\NC}}
\nc{\rbpltree}{\calt^r}
\nc{\scs}[1]{\scriptstyle{#1}} \nc{\mrm}[1]{{\rm #1}}
\nc{\dirlim}{\displaystyle{\lim_{\longrightarrow}}\,}
\nc{\invlim}{\displaystyle{\lim_{\longleftarrow}}\,}
\nc{\mvp}{\vspace{0.5cm}} \nc{\svp}{\vspace{2cm}}
\nc{\vp}{\vspace{8cm}} \nc{\proofbegin}{\noindent{\bf Proof: }}
\nc{\proofend}{$\blacksquare$ \vspace{0.5cm}}
\nc{\freerbpl}{{F^{\mathrm RBPL}}}
\nc{\sha}{{\mbox{\cyr X}}}  
\nc{\ncsha}{{\mbox{\cyr X}^{\mathrm NC}}} \nc{\ncshao}{{\mbox{\cyr
X}^{\mathrm NC,\,0}}}
\nc{\shpr}{\diamond}    
\nc{\shprm}{\overline{\diamond}}    
\nc{\shpro}{\diamond^0}    
\nc{\shprr}{\diamond^r}     
\nc{\shpra}{\overline{\diamond}^r}
\nc{\shpru}{\check{\diamond}} \nc{\catpr}{\diamond_l}
\nc{\rcatpr}{\diamond_r} \nc{\lapr}{\diamond_a}
\nc{\sqcupm}{\ot}
\nc{\lepr}{\diamond_e} \nc{\vep}{\varepsilon} \nc{\labs}{\mid\!}
\nc{\rabs}{\!\mid} \nc{\hsha}{\widehat{\sha}}
\nc{\lsha}{\stackrel{\leftarrow}{\sha}}
\nc{\rsha}{\stackrel{\rightarrow}{\sha}} \nc{\lc}{\lfloor}
\nc{\rc}{\rfloor}
\nc{\tpr}{\sqcup}
\nc{\nctpr}{\vee}
\nc{\plpr}{\star}
\nc{\rbplpr}{\bar{\plpr}}
\nc{\sqmon}[1]{\langle #1\rangle}
\nc{\forest}{\calf}
\nc{\altx}{\Lambda_X} \nc{\vecT}{\vec{T}} \nc{\onetree}{\bullet}
\nc{\Ao}{\check{A}}
\nc{\seta}{\underline{\Ao}}
\nc{\deltaa}{\overline{\delta}}
\nc{\trho}{\tilde{\rho}}

\nc{\rpr}{\circ}
\nc{\dpr}{{\tiny\diamond}}
\nc{\rprpm}{{\rpr}}

\nc{\mmbox}[1]{\mbox{\ #1\ }} \nc{\ann}{\mrm{ann}}
\nc{\Aut}{\mrm{Aut}} \nc{\can}{\mrm{can}}
\nc{\twoalg}{{two-sided algebra}\xspace}
\nc{\colim}{\mrm{colim}}
\nc{\Cont}{\mrm{Cont}} \nc{\rchar}{\mrm{char}}
\nc{\cok}{\mrm{coker}} \nc{\dtf}{{R-{\rm tf}}} \nc{\dtor}{{R-{\rm
tor}}}

\nc{\depth}{{\mrm d}}
\nc{\Div}{{\mrm Div}} \nc{\End}{\mrm{End}} \nc{\Ext}{\mrm{Ext}}
\nc{\Fil}{\mrm{Fil}} \nc{\Frob}{\mrm{Frob}} \nc{\Gal}{\mrm{Gal}}
\nc{\GL}{\mrm{GL}} \nc{\Hom}{\mrm{Hom}} \nc{\hsr}{\mrm{H}}
\nc{\hpol}{\mrm{HP}} \nc{\id}{\mrm{id}} \nc{\im}{\mrm{im}}
\nc{\incl}{\mrm{incl}} \nc{\length}{\mrm{length}}
\nc{\LR}{\mrm{LR}} \nc{\mchar}{\rm char} \nc{\NC}{\mrm{NC}}
\nc{\mpart}{\mrm{part}} \nc{\pl}{\mrm{PL}}
\nc{\ql}{{\QQ_\ell}} \nc{\qp}{{\QQ_p}}
\nc{\rank}{\mrm{rank}} \nc{\rba}{\rm{RBA }} \nc{\rbas}{\rm{RBAs }}
\nc{\rbpl}{\mrm{RBPL}}
\nc{\rbw}{\rm{RBW }} \nc{\rbws}{\rm{RBWs }} \nc{\rcot}{\mrm{cot}}
\nc{\rest}{\rm{controlled}\xspace}
\nc{\rdef}{\mrm{def}} \nc{\rdiv}{{\rm div}} \nc{\rtf}{{\rm tf}}
\nc{\rtor}{{\rm tor}} \nc{\res}{\mrm{res}} \nc{\SL}{\mrm{SL}}
\nc{\Spec}{\mrm{Spec}} \nc{\tor}{\mrm{tor}} \nc{\Tr}{\mrm{Tr}}
\nc{\mtr}{\mrm{sk}}

\nc{\ab}{\mathbf{Ab}} \nc{\Alg}{\mathbf{Alg}}
\nc{\Algo}{\mathbf{Alg}^0} \nc{\Bax}{\mathbf{Bax}}
\nc{\Baxo}{\mathbf{Bax}^0} \nc{\RB}{\mathbf{RB}}
\nc{\RBo}{\mathbf{RB}^0} \nc{\BRB}{\mathbf{RB}}
\nc{\Dend}{\mathbf{DD}} \nc{\bfk}{{\bf k}} \nc{\bfone}{{\bf 1}}
\nc{\base}[1]{{a_{#1}}} \nc{\detail}{\marginpar{\bf More detail}
    \noindent{\bf Need more detail!}
    \svp}
\nc{\Diff}{\mathbf{Diff}} \nc{\gap}{\marginpar{\bf
Incomplete}\noindent{\bf Incomplete!!}
    \svp}
\nc{\FMod}{\mathbf{FMod}} \nc{\mset}{\mathbf{MSet}}
\nc{\rb}{\mathrm{RB}} \nc{\Int}{\mathbf{Int}}
\nc{\Mon}{\mathbf{Mon}}
\nc{\remarks}{\noindent{\bf Remarks: }}
\nc{\OS}{\mathbf{OS}} 
\nc{\Rep}{\mathbf{Rep}}
\nc{\Rings}{\mathbf{Rings}} \nc{\Sets}{\mathbf{Sets}}
\nc{\DT}{\mathbf{DT}}

\nc{\BA}{{\mathbb A}} \nc{\CC}{{\mathbb C}} \nc{\DD}{{\mathbb D}}
\nc{\EE}{{\mathbb E}} \nc{\FF}{{\mathbb F}} \nc{\GG}{{\mathbb G}}
\nc{\HH}{{\mathbb H}} \nc{\LL}{{\mathbb L}} \nc{\NN}{{\mathbb N}}
\nc{\QQ}{{\mathbb Q}} \nc{\RR}{{\mathbb R}} \nc{\BS}{{\mathbb{S}}} \nc{\TT}{{\mathbb T}}
\nc{\VV}{{\mathbb V}} \nc{\ZZ}{{\mathbb Z}}

\nc{\calao}{{\mathcal A}} \nc{\cala}{{\mathcal A}}
\nc{\calc}{{\mathcal C}} \nc{\cald}{{\mathcal D}}
\nc{\cale}{{\mathcal E}} \nc{\calf}{{\mathcal F}}
\nc{\calfr}{{{\mathcal F}^{\,r}}} \nc{\calfo}{{\mathcal F}^0}
\nc{\calfro}{{\mathcal F}^{\,r,0}} \nc{\oF}{\overline{F}}
\nc{\calg}{{\mathcal G}} \nc{\calh}{{\mathcal H}}
\nc{\cali}{{\mathcal I}} \nc{\calj}{{\mathcal J}}
\nc{\call}{{\mathcal L}} \nc{\calm}{{\mathcal M}}
\nc{\caln}{{\mathcal N}} \nc{\calo}{{\mathcal O}}
\nc{\calp}{{\mathcal P}} \nc{\calq}{{\mathcal Q}} \nc{\calr}{{\mathcal R}}
\nc{\calt}{{\mathcal T}} \nc{\caltr}{{\mathcal T}^{\,r}}
\nc{\calu}{{\mathcal U}} \nc{\calv}{{\mathcal V}}
\nc{\calw}{{\mathcal W}} \nc{\calx}{{\mathcal X}}
\nc{\CA}{\mathcal{A}}

\nc{\fraka}{{\mathfrak a}} \nc{\frakB}{{\mathfrak B}}
\nc{\frakb}{{\mathfrak b}} \nc{\frakd}{{\mathfrak d}}
\nc{\oD}{\overline{D}}
\nc{\frakF}{{\mathfrak F}} \nc{\frakg}{{\mathfrak g}}
\nc{\frakm}{{\mathfrak m}} \nc{\frakM}{{\mathfrak M}}
\nc{\frakMo}{{\mathfrak M}^0} \nc{\frakp}{{\mathfrak p}}
\nc{\frakS}{{\mathfrak S}} \nc{\frakSo}{{\mathfrak S}^0}
\nc{\fraks}{{\mathfrak s}} \nc{\os}{\overline{\fraks}}
\nc{\frakT}{{\mathfrak T}}
\nc{\oT}{\overline{T}}
\nc{\frakX}{{\mathfrak X}} \nc{\frakXo}{{\mathfrak X}^0}
\nc{\frakx}{{\mathbf x}}
\nc{\frakTx}{\frakT}      
\nc{\frakTa}{\frakT^a}        
\nc{\frakTxo}{\frakTx^0}   
\nc{\caltao}{\calt^{a,0}}   
\nc{\ox}{\overline{\frakx}} \nc{\fraky}{{\mathfrak y}}
\nc{\frakz}{{\mathfrak z}} \nc{\oX}{\overline{X}}

\font\cyr=wncyr10

\nc{\redtext}[1]{\textcolor{red}{#1}}

\makeatletter
\g@addto@macro{\endabstract}{\@setabstract}
\newcommand{\authorfootnotes}{\renewcommand\thefootnote{\@fnsymbol\c@footnote}}%
\makeatother
\begin{document}
\begin{center}
  \LARGE
\textbf{Biderivations and commutative post-Lie algebra structure on the Schr\"{o}dinger-Virasoro Lie algebra }

  \normalsize
  \authorfootnotes
Xiaomin Tang   \footnote{Corresponding author: {\it X. Tang. Email:} x.m.tang@163.com}
\par \bigskip

   \textsuperscript{1}Department of Mathematics, Heilongjiang University,
Harbin, 150080, P. R. China   \par

\end{center}


\begin{abstract}

In this paper, we characterize the biderivations of the Schr\"{o}dinger-Virasoro Lie algebra. We obtain a class of non-inner and non-skewsymmetric biderivations. As an application, we characterize the commutative post-Lie algebra structures on the Schr\"{o}dinger-Virasoro Lie algebra.

\vspace{2mm}

\noindent{\it Keywords:} biderivation, skewsymmetric, Schr\"{o}dinger-Virasoro Lie algebra, post-Lie algebra

\noindent{\it AMS subject classifications:} 17B05, 17B40, 17B65.

\end{abstract}

\setcounter{section}{0}
{\ }

 \baselineskip=20pt

\section{Introduction and  preliminary results}

The Schr\"{o}dinger-Virasoro Lie algebra is an infinite-dimensional Lie algebra that was introduced in \cite{Hen} in the context of non-equilibrium
statistical physics. It contains as subalgebras both the Lie algebra of invariance of the free Schr\"{o}dinger equation and the central charge-free Virasoro
algebra ${\rm{Vect}}(S^1)$. To be precise, for $\varepsilon\in\{0,\frac{1}{2}\}$, the Schr\"{o}dinger-Virasoro Lie algebra $\mathcal{SV}(\varepsilon)$ is a Lie algebra with the $\mathbb{C}$ basis
$$\{L_i, Y_j, M_i| i\in \mathbb{Z}, j\in \varepsilon+ \mathbb{Z}\}$$
and Lie brackets
\begin{eqnarray*}
&[L_m,L_n]=(m-n)L_{m+n},\\
&[L_m,Y_n]=(\frac{1}{2}m-n)Y_{m+n},\\
&[L_m,M_n]=-nM_{m+n},\\
&[Y_m,Y_n]=(m-n)M_{m+n},\\
&[Y_m,M_n]=[M_m,M_n]=0.
\end{eqnarray*}
The Lie algebra $\mathcal{SV}(\frac{1}{2})$ is called the original Schr\"{o}dinger-Virasoro Lie algebra, and $\mathcal{SV}(0)$ is called the twisted Schr\"{o}dinger-Virasoro Lie algebra. Recently, the theory of the structures and representations of both the original and twisted Schr\"{o}dinger-Virasoro Lie algebras has
been investigated in a series of studies. For instance, the Lie bialgebra structures, derivations, automorphisms, $2$-cocycles, vertex algebra representations and Whittaker modules were investigated in \cite{HanS,Lisu2008,Lisu2009,peib,R,tansb,Unte,Zhangtan}.  In particular, to determine the form of each commuting map on the twisted Schr\"{o}dinger-Virasoro Lie algebra, the authors of \cite{WD1} describe the skewsymmetric biderivations of $\mathcal{SV}(0)$. Now let us review some details regarding derivations and biderivations.

Let $(A,\cdot)$ be an (associative or non-associative) algebra on a field. A linear map $d: A \rightarrow A$ is called a derivation of $A$ if
$$d(x\cdot y) = d(x)\cdot y + x\cdot d(y),$$
for all $x, y\in A$. A bilinear map $f: A\times A \rightarrow A$ is called a biderivation of  $A$ if it
is a derivation with respect to both components, which indicates that
$$f(x\cdot y,z)= x\cdot f(y,z)+ f(x,z)\cdot y  \ \text{and} \ f(x,y\cdot z)=f(x,y)\cdot z+ y\cdot f(x,z)$$
for all $x,y,z\in A$. Note that $f$ is called skewsymmetric if $f(x,y)=-f(y,x)$ for all $x,y\in A$.

Biderivations are a subject of research in various areas \cite{Bre1995,Chen2016,Du2013,Gho2013,Hanw,tang2016,WD1,WD3,WD2}. In \cite{Bre1995}, Bre$\breve{s}$ar et al. showed that all biderivations on commutative prime rings are inner biderivations and determined the biderivations of semiprime rings. This theorem has proved to be useful in the study of commutating maps. More details regarding commuting maps, biderivations and their
generalizations can be found in the survey article \cite{Bre3}.

If $(A,\cdot)$ is an associative algebra, we further let $[x, y] = x\cdot y- y\cdot x$ be the commutator of the elements $x, y\in A$. Then, $(A,[,])$ forms a Lie algebra, which is called the compatible Lie algebra with $(A,\cdot)$.  It is not difficult to verify the fundamental truth that every derivation (resp., biderivation) of an associative algebra $(A,\cdot)$ naturally becomes a derivation (resp., biderivation) of the compatible Lie algebra $(A,[,])$ with $(A,\cdot)$.
Thus, the study of a derivation or biderivation of a Lie algebra should be more general than that of an associative algebra.
The notation of biderivations of Lie algebras was introduced in 2011 \cite{WD3}, well after the introduction of the notation of biderivations of associative algebras.  After \cite{WD3}, many authors began studying (super-)biderivaitons of some Lie (super-)algebras, such as \cite{Chen2016,Hanw,tang2016,WD1,WD2}.

For convenience, we now review the concept of a biderivation of a Lie algebra as
follows. For an arbitrary Lie algebra $(L,[,])$, we recall that a bilinear map $f : L\times L \rightarrow L$ is a
biderivation of $L$ if it is a derivation with respect to both components. To be more precise, one has
\begin{eqnarray}
f([x,y],z)=[x,f(y,z)]+[f(x,z),y], \label{2der}\\
f(x,[y,z])=[f(x,y),z]+[y,f(x,z)] \label{1der}
\end{eqnarray}
for all $x, y, z\in L$.
For a complex number $\lambda$, we define a bilinear map $f: L\times L\rightarrow L$ given by $f(x,y)=\lambda [x,y]$. Then, it is easy to verify that $f$
is a biderivation of $L$. We refer to such a biderivation as an inner biderivation.

The authors of \cite{WD1} prove that if a biderivation of $\mathcal{SV}(0)$ is skewsymmetric, then it has to be inner. We first give an example to show that there is  a non-inner and non-skewsymmetric biderivation of $\mathcal{SV}(\varepsilon)$.  Note that the Schr\"{o}dinger-Virasoro Lie algebra $\mathcal{SV}(\varepsilon)$ has the following decomposition of subspaces: $\mathcal{SV}(\epsilon)=\mathfrak{L}\oplus\mathfrak{Y}\oplus\mathfrak{M}$, where
$$
\mathfrak{L}=\bigoplus_{i\in \mathbb{Z}} \mathbb{C}L_i, \  \mathfrak{Y}=\bigoplus_{i\in \varepsilon+\mathbb{Z}} \mathbb{C}Y_i, \ \mathfrak{M}=\bigoplus_{i\in \mathbb{Z}} \mathbb{C}M_i.
$$
\begin{example} \label{lihai}
Suppose that  $f: \mathcal{SV}(\varepsilon)\times \mathcal{SV}(\varepsilon)\rightarrow \mathcal{SV}(\varepsilon)$ is a bilinear map determined by
$$
f(L_m,L_n)=2017M_{m+n+2016}
$$
for all $m,n\in \mathbb{Z}$, and $f(x,y)=0$ if either of $x, y$ is contained in $\mathfrak{Y}\cup\mathfrak{M}$.  It is easy to verify that
$f$ is a biderivation of $ \mathcal{SV}$. However, we find that the biderivation is non-inner and non-skewsymmetric.
\end{example}

Recall that a linear map $\phi: L\rightarrow L$ for a Lie algebra $L$ is called a derivation if it satisfies
$\phi([x,y])=[\phi(x),y]+[x,\phi(y)]
$
for all $x, y\in L$.
For $x\in L$, it is easy to see that $\phi_x:L\rightarrow L, y\mapsto {\rm ad} x(y)=[x,y], $ for all $y\in L$ is a derivation of $L$, which is called an inner derivation.

\begin{lemma}\cite{R}\label{innerLie}
Every derivation of $\mathcal{SV}(\varepsilon)$ is of the following form:
$$
{\rm ad} x+a D_1+b D_2 +cD_3
$$
for some $x\in \mathcal{SV}(\varepsilon)$ and $a,b,c\in \mathbb{C}$, where $D_i,i=1,2,3$ are outer derivations of $\mathcal{SV}(\varepsilon)$, which is defined by
\begin{eqnarray*}
&D_1(L_m)=M_m,\ D_1(Y_m)=D_1(M_m)=0,\\
&D_2(L_m)=mM_m,\ D_2(Y_m)=D_2(M_m)=0,\\
&D_3(L_m)=0,\ D_3(Y_m)=Y_m, \ D_3(M_m)=2M_m.
\end{eqnarray*}
for all $m\in \mathbb{Z}$.
\end{lemma}

Although we work under the complex number field $\mathbb{C}$ in this study, this field also works with any algebraically closed field of characteristic zero.
We first characterize the biderivation of $\mathcal{SV}(\varepsilon)$ without the skewsymmetry condition and then present an application for a commutative post-Lie algebra.

\section{ Formal calculus }
In this section, we prove some results of formal calculus, which will be useful in our main proof.

\begin{proposition} \label{prop1}
Suppose that $k_i^{(m)}$ and $h_i^{(m)}$ are numbers satisfying
\begin{eqnarray}
&&(i-n)k_i^{(m)}=(2m-n-i)h_{n-m+i}^{(n)}, \label{abcd1}
\end{eqnarray}
for all $m,n,i\in \mathbb{Z}$. Then, there is $\lambda\in \mathbb{C}$ such that
$
k_i^{(m)}=h_{i}^{(m)}=\delta_{m,i} \lambda, \ \forall m,i\in \mathbb{Z},
$
where $\delta_{m,i}$ is the Kronecker delta.
\end{proposition}

\begin{proof}
For any $m,n$ with $m\neq n$, by taking $i= 2m-n, n,$ and $m $  in (\ref{abcd1}), respectively, we have
 \begin{equation}\label{kkk}
 k_{2m-n}^{(m)}=h_{2n-m}^{(n)}=0, \ k_{m}^{(m)}=h_{n}^{(n)}, \ \forall m, n\in \mathbb{Z} \ \text {with}\ m\neq n.
 \end{equation}
 Let $m, n$ run over all integers with $m\neq n$. We conclude based on (\ref{kkk}) that
 \begin{equation}\label{kkklll}
 k_{i}^{(m)}=0, h_{j}^{(n)}=0,  \forall i\neq m, j\neq n,
 \end{equation}
 and
\begin{equation}\label{kkkrrr}
k_{n}^{(n)}=h_{n}^{(n)}=\ h_{0}^{(0)},\ \forall n\in \mathbb{Z}.
\end{equation}
  Setting $\ h_{0}^{(0)}=\lambda$ completes the proof.
\end{proof}

\begin{proposition}\label{prop2}
 Suppose that $t_i^{(m)}$ and $g_i^{(m)}$ are numbers satisfying
\begin{eqnarray}
&&(i-\frac{n}{2})t_i^{(m)}=(\frac{3m}{2}-n-i)g_{n-m+i}^{(n)} \label{abcd2}
\end{eqnarray}
for all $m,n,i\in \mathbb{Z}$. Then,
$
t_i^{(m)}=g_{i}^{(m)}=0, \ \forall m,i\in \mathbb{Z}.
$
\end{proposition}

\begin{proof}

The proof is divided into the following steps.

{\it Step 1.} Let $n=0$ in (\ref{abcd2}). Then
\begin{equation}\label{jintian1}
it_i^{(m)}=(\frac{3m}{2}-i)g_{-m+i}^{(0)}.
\end{equation}
Let $i=0$ in the above equation. Then $g_{-m}^{(0)}=0$ for every $m\neq 0$. In particular, $g_{-m+i}^{(0)}=0$ if $-m+i\neq 0$. This, together with (\ref{jintian1}), yields $it_i^{(m)}=0$ for all $i\neq m$. Thus, $t_i^{(m)}=0$ for all $i\neq 0, m$. Similarly, by letting $n=0$ and $i=-n$ in (\ref{jintian1}), we deduce that $g_i^{(m)}=0$ for all $i\neq 0, m$.

{\it Step 2.}  Take $i=m$ in (\ref{abcd2}). We deduce
\begin{equation}\label{jintian2}
(m-\frac{n}{2})t_m^{(m)}=(\frac{m}{2}-n)g_{n}^{(n)}.
\end{equation}
By letting $n=0$ in (\ref{jintian2}), one obtains $mt_m^{(m)}=\frac{m}{2}g_{0}^{(0)}$. Thus, $t_m^{(m)}=\frac{1}{2}g_{0}^{(0)}$ for all $m\neq 0$.  Similarly,
by letting $m=0$ in (\ref{jintian2}), we have that $g_n^{(n)}=\frac{1}{2}t_{0}^{(0)}$ for all $n\neq 0$.

{\it Step 3.} Now let $m=i=2$ and $n=1$ in (\ref{abcd2}). Then, $t_{2}^{(2)}=0$. This, with Step 2, gives that $0=t_2^{(2)}=t_m^{(m)}=\frac{1}{2}g_{0}^{(0)}$ for all $m\neq 0$. Similarly, by letting $i=m=1$ and $n=2$ in (\ref{abcd2}), we have $g_{2}^{(2)}=0$. Again, we use Step 2 and deduce that  $0=g_n^{(n)}=\frac{1}{2}t_{0}^{(0)}$ for all $n\neq 0$.

{\it Step 4.} By Steps 1-3, to finish the proof, it is sufficient to prove that $t_0^{(n)}=g_0^{(n)}=0$ for all $n\neq 0$. In fact, by letting $i=0$ and $n=-m\neq 0$ in (\ref{abcd2}), we obtain $t_0^{(m)}=5g_{-2m}^{(-m)}$. Because $-2m\neq 0$, we see that $g_{-2m}^{(-m)}=0$ and, thus, $t_0^{(m)}=0$ for all $m\neq 0$.  Similarly, by letting $i=m-n$ and $m=-n\neq 0$ in (\ref{abcd2}), we deduce that $g_0^{(n)}=0$ for all $n\neq 0$. The proof is completed.
\end{proof}

\begin{proposition} \label{prop3}
Suppose that $s_i^{(m)}, e_i^{(m)}, i,m\in \mathbb{Z}$ are numbers and that $\rho_1, \rho_2, \theta_1, \theta_2$ are linear complex-valued functions on $\mathcal{SV}(\varepsilon)$ satisfying
\begin{eqnarray}
&&i s_{i}^{(m)}=-(n-m+i) e_{n-m+i}^{(n)},  \ {\text for } \ i\neq 0, m-n, \label{mingtian1}\\
&& \rho_1(L_m)+ n \rho_2(L_m) =-(n-m) e_{n-m}^{(n)}, \label{mingtian2}\\
&& \theta_1(L_n)+ m \theta_2(L_n) =(m-n) s_{m-n}^{(m)}\label{mingtian3}
\end{eqnarray}
for all $m,n,i\in \mathbb{Z}$. Then, there is a set of complex numbers $\Omega=\{\mu_i\in \mathbb{C}|i\in \mathbb{Z}\}$ such that
\begin{eqnarray}
&&s_{m+k}^{(m)}=-e_{m+k}^{(m)}=\frac{\mu_k}{m+k}, \ \forall k\in \mathbb{Z}\setminus\{ -m\},\label{jintian4} \\
&&\rho_1(L_m)=\theta_1(L_m)=\mu_{-m}, \label{jintian5} \\
&&\rho_2(L_m)=\theta_2(L_m)=0. \label{jintian6}
\end{eqnarray}
\end{proposition}

\begin{proof}
Let $i=m+k$ in (\ref{mingtian1}) with $k\neq -m$ and $k\neq -n$. Then,
\begin{equation}\label{mingtian4}
(m+k)s_{m+k}^{(m)}=-(n+k)e_{n+k}^{(n)}.
\end{equation}
Let $m=n=1-k$ in (\ref{mingtian4}). Then $s_{1}^{(1-k)}=-e_{1}^{(1-k)}$. Denote $\mu_k=s_{1}^{(1-k)}$.
By using (\ref{mingtian4}) with $n=1-k$, we deduce that $s_{m+k}^{(m)}=\frac{\mu_k}{m+k}$ for all $k\in \mathbb{Z}\setminus\{ -m\}$.
Similarly, by using (\ref{mingtian4}) with $m=1-k$, we see that $e_{n+k}^{(n)}=-\frac{\mu_k}{n+k}$ for all $k\in \mathbb{Z}\setminus\{ -n\}$.
This proves that (\ref{jintian4}) holds.  Note that $(n-m)e_{n-m}^{(n)}=-\mu_{-m}$. It follows from (\ref{mingtian2}) that
$$
\rho_1(L_m)+ n \rho_2(L_m) =\mu_{-m}.
$$
This indicates that $\rho_1(L_m)=\mu_{-m}$ and $\rho_2(L_m)=0$. Similarly, by  $(m-n)s_{m-n}^{(m)}=\mu_{-n}$ and (\ref{mingtian3}), we deduce that
$\theta_1(L_n)=\mu_{-n}$ and $\theta_2(L_n)=0$. The proof is completed.
\end{proof}

\begin{proposition}\label{prop4}
Suppose that $q_i^{(n)}$ are numbers satisfying
\begin{equation}\label{jintian3}
(\frac{m}{2}-i)q_i^{(n)}=0
\end{equation}
for all $m,n,i\in \mathbb{Z}$ with $i\neq n$ and $i\neq m-n$. Then,
$q_i^{(n)}=0$ for all $i,n\in \mathbb{Z}$ with $i\neq n$.
\end{proposition}

\begin{proof}
For any fixed $i,n\in \mathbb{Z}$ with $i\neq n$, it is easy to see that there exist $m_1,m_2\in \mathbb{Z}$ satisfying
$$
m_1\neq m_2,\ \ i\notin \{m_1-n,  m_2-n\}.
$$
Let $m=m_1$ and $m=m_2$ in (\ref{jintian3}). We obtain, respectively,
\begin{equation*}
(\frac{m_1}{2}-i)q_i^{(n)}=0 \ \text{and}\ (\frac{m_2}{2}-i)q_i^{(n)}=0.
\end{equation*}
It follows from $m_1\neq m_2$ that $q_i^{(n)}=0$. The proof is completed.
\end{proof}

\section{Biderivation of Schr\"{o}dinger-Virasoro Lie algebra}

In this section, we assume that $f$ is a biderivation of Schr\"{o}dinger-Virasoro Lie algebra $\mathcal{SV}(\varepsilon)$.
\begin{lemma}\label{phipsi}
There are two linear maps $\phi$ and $\psi$ from $\mathcal{SV}(\varepsilon)$ into itself such that
\begin{eqnarray}\label{lr}
f(x,y)&=&\rho_1(x) D_1(y)+\rho_2(x)D_2(y)+\rho_3(x)D_3(y)+[\phi(x), y],\\
&=&\theta_1(y) D_1(x)+\theta_2(y)D_2(x)+\theta_3(y)D_3(x)+[x, \psi(y)]
\end{eqnarray}
for all $x,y\in \mathcal{SV}$, where $\rho_i(x), \theta_i(x)$ are linear complex-valued functions on $\mathcal{SV}(\varepsilon)$ and $D_i$ is given by Lemma \ref{innerLie}, for  $i=1,2,3$.
\end{lemma}
\begin{proof}
For the biderivation $f$ of $\mathcal{SV}(\varepsilon)$ and a fixed element $x\in \mathcal{SV}(\varepsilon)$, we define a map $\phi_x: \mathcal{SV}(\varepsilon)\rightarrow \mathcal{SV}(\varepsilon)$ given by $\phi_x(y)=f(x,y)$. Then, we know from (\ref{1der}) that $\phi_x$ is a derivation of $\mathcal{SV}(\varepsilon)$. By Lemma \ref{innerLie}, there exist some complex-valued functions $\rho_1, \rho_2, \rho_3$ on $\mathcal{SV}(\varepsilon)$ and a map $\phi$ from $\mathcal{SV}(\varepsilon)$ into itself such that $\phi_x =\rho_1(x) D_1+\rho_2(x) D_2 +\rho_3 (x) D_3+{\rm ad}\phi(x)$.
Namely, $f(x,y)=\rho_1(x) D_1(y)+\rho_2(x)D_2(y)+\rho_3(x) D_3(y)+[\phi(x), y]$. Because $f$ is bilinear, one has that $\rho_1, \rho_2, \rho_3$ and $\phi$ are linear. Similarly, we define a map $\psi_z$ from $\mathcal{SV}(\varepsilon)$ into itself given by $\psi_z(y)=f(y, z)$ for all $y\in \mathcal{SV}(\varepsilon)$. We can obtain linear complex-valued functions $\theta_1,\theta_2,\theta_3$ on $\mathcal{SV}(\varepsilon)$ and a linear map $\psi$ from $\mathcal{SV}(\varepsilon)$ into itself such that $f(x,y)=\theta_1(y) D_1(x)+\theta_2(y)D_2(x)+\theta_3(y)D_3(x)+{\rm ad}(-\psi (y))(x)=\theta_1(y) D_1(x)+\theta_2(y)D_2(x)+\theta_3(y)D_3(x)+[x, \psi(y)]$. The proof is completed.
\end{proof}

Below, we discuss only the case in which $\varepsilon=0$, i.e., the twisted Schr\"{o}dinger-Virasoro Lie algebra. It is easy to verify that the method and result are also valid for the case in which $\varepsilon=\frac{1}{2}$.

\begin{lemma}\label{LMLN}
Let $\phi$ and $\psi$ be defined by Lemma \ref{phipsi}. Then, there are $\lambda,s_{0}^{(m)},e_{0}^{(m)}\in \mathbb{C}$ and a set $\Omega=\{\mu_k\in \mathbb{C}| k\in \mathbb{Z}\}$ with $|\Omega|<+\infty$ such that
\begin{eqnarray}
\phi(L_m)=\lambda L_m +\sum_{k\in \mathbb{Z}\setminus{\{-m\}}} \frac{\mu_k}{m+k}M_{m+k}+s_{0}^{(m)}M_0, \label{tang21}\\
\psi(L_m)=\lambda L_m +\sum_{k\in \mathbb{Z}\setminus{\{-m\}}} \frac{-\mu_k}{m+k}M_{m+k}+e_{0}^{(m)}M_0, \label{tang22}
\end{eqnarray}
for any $m\in \mathbb{Z}$.
\end{lemma}

\begin{proof}
For any $n\in \mathbb{Z}$, let
\begin{eqnarray}
\phi(L_n)=\sum_{i\in \mathbb{Z}} k_{i}^{(n)}L_i+\sum_{i\in \mathbb{Z}} t_{i}^{(n)}Y_i+\sum_{i\in \mathbb{Z}} s_{i}^{(n)}M_i, \label{ee1}\\
\psi(L_{n})=\sum_{i\in \mathbb{Z}} h_{i}^{(n)}L_i+\sum_{i\in \mathbb{Z}} g_{i}^{(n)}Y_i+\sum_{i\in \mathbb{Z}} e_{i}^{(n)}M_i, \label{ee3}
\end{eqnarray}
where $k_{i}^{(n)}, t_{i}^{(n)},   s_{i}^{(n)}, h_{i}^{(n)}, g_{i}^{(n)},  e_{i}^{(n)} \in \mathbb{C}$ for every $i\in \mathbb{Z}$.
Therefore, we have
\begin{equation}\label{philmln}
[\phi (L_m), L_n]=\sum_{i\in \mathbb{Z}} (i-n)k_{i}^{(m)}L_{n+i}+\sum_{i\in \mathbb{Z}}(i-\frac{n}{2})t_{i}^{(m)}Y_{n+i}+\sum_{i\in \mathbb{Z}} i s_{i}^{(m)}M_{n+i},
\end{equation}
\begin{eqnarray}
[L_m,\psi(L_n)]&=&\sum_{i\in \mathbb{Z}} (m-i)h_{i}^{(n)}L_{m+i}+
\sum_{i\in \mathbb{Z}}(\frac{m}{2}-i)g_{i}^{(n)}Y_{m+i}-\sum_{i\in \mathbb{Z}} i e_{i}^{(n)}M_{m+i} \nonumber\\
&=&\sum_{i\in \mathbb{Z}} (2m-n-i)h_{n-m+i}^{(n)}L_{n+i}+
\sum_{i\in \mathbb{Z}}(\frac{3m}{2}-n-i)g_{n-m+i}^{(n)}Y_{n+i}\label{lmpsiln}\\
&&-\sum_{i\in \mathbb{Z}} (n-m+i) e_{n-m+i}^{(n)}M_{n+i}. \nonumber
\end{eqnarray}
By Lemma (\ref{phipsi}),
$$
f(L_m,L_n)=(\rho_1(L_m)+ n \rho_2(L_m))M_n+[\phi (L_m), L_n]=(\theta_1(L_n)+ m \theta_2(L_n))M_m+[L_m,\psi(L_n)].
$$
 This, together with (\ref{philmln}) and (\ref{lmpsiln}), implies that Equations (\ref{abcd1}),(\ref{abcd2}), (\ref{mingtian1}), (\ref{mingtian2}) and (\ref{mingtian3}) are satisfied.  By Propositions \ref{prop1}-\ref{prop3}, we know that $\phi(L_n)$ and $\psi(L_n)$ satisfy (\ref{tang21}) and (\ref{tang22}), respectively. Note that the expansion of $\phi(L_m)$ is finite sum for any integer $m$, so $\Omega$ must contain only finitely many nonzero numbers. In other words, $|\Omega|<+\infty$.
 The proof is completed.
\end{proof}

As a corollary of Proposition \ref{prop3}, we obtain (\ref{jintian5}) and (\ref{jintian6});
namely, $\rho_1(L_m)=\theta_1(L_m)=\mu_{-m}$ and $\rho_2(L_m)=\theta_2(L_m)=0$. This allows us to describe $f(L_m,L_n)$ as follows.
\begin{eqnarray}
f(L_m,L_n)&=&(\rho_1(L_m)+ n \rho_2(L_m))M_n+[\phi (L_m), L_n] \nonumber\\
&=& \mu_{-m}M_n+\lambda[L_m,L_n]+\sum_{k\in \mathbb{Z}\setminus{\{-m\}}} \frac{\mu_k}{m+k}[M_{m+k},L_n] \nonumber\\
&=& \mu_{-m}M_n+\lambda[L_m,L_n]+\sum_{k\in \mathbb{Z}\setminus{\{-m\}}} \mu_k M_{m+n+k} \nonumber\\
&=& \lambda[L_m,L_n]+\sum_{k\in \mathbb{Z}} \mu_k M_{m+n+k}. \label{laop1}
\end{eqnarray}

\begin{lemma}\label{yform}
Let $\phi$ and $\psi$ be defined by Lemma \ref{phipsi}. Then, there are $c_0^{(n)}, r_0^{(n)}\in \mathbb{C}$ such that
$$
\phi(Y_n)=\lambda Y_n+c_0^{(n)}M_0,\ \ \psi(Y_n)=\lambda Y_n+r_0^{(n)}M_0, \ \forall n\in \mathbb{Z},
$$
where $\lambda$ is the same as in Lemma \ref{LMLN}. Furthermore,
$$
\rho_i(Y_n)=\theta_i(Y_n)=\rho_3(L_n)=\theta_3(L_n)=0, i=1,2,3,  \ \forall n\in \mathbb{Z}.
$$
\end{lemma}

\begin{proof}
For any $n\in \mathbb{Z}$, let
\begin{eqnarray}
\phi(Y_n)=\sum_{i\in \mathbb{Z}} a_{i}^{(n)}L_i+\sum_{i\in \mathbb{Z}} b_{i}^{(n)}Y_i+\sum_{i\in \mathbb{Z}} c_{i}^{(n)}M_i, \label{ee23}\\
\psi(Y_{n})=\sum_{i\in \mathbb{Z}} p_{i}^{(n)}L_i+\sum_{i\in \mathbb{Z}} q_{i}^{(n)}Y_i+\sum_{i\in \mathbb{Z}} r_{i}^{(n)}M_i, \label{ee33}
\end{eqnarray}
where $a_{i}^{(n)}, b_{i}^{(n)},   c_{i}^{(n)}, p_{i}^{(n)}, q_{i}^{(n)},  r_{i}^{(n)} \in \mathbb{C}$ for every $i\in \mathbb{Z}$. Due to Lemma \ref{LMLN}, by a simple computation, one has
\begin{equation}\label{p123}
[\phi (L_m), Y_n]=\lambda(\frac{m}{2}-n)Y_{m+n},
\end{equation}
and
\begin{eqnarray}\label{p124}
[L_m,\psi(Y_n)]&=&\sum_{i\in \mathbb{Z}} (m-i)p_{i}^{(n)}L_{m+i}+
\sum_{i\in \mathbb{Z}}(\frac{m}{2}-i)q_{i}^{(n)}Y_{m+i}-\sum_{i\in \mathbb{Z}} i r_{i}^{(n)}M_{m+i}.
\end{eqnarray}
By Lemma \ref{phipsi}, we have
\begin{equation}\label{p125}
\rho_3(L_m) Y_n +[\phi (L_m), Y_n]=(\theta_1(Y_n)+ m \theta_2(Y_n))M_m+[L_m,\psi(Y_n)].
\end{equation}
This, together with (\ref{p123}) and (\ref{p124}), yields that (\ref{jintian3}) holds and $(m-i)p_{i}^{(n)}= i r_{i}^{(n)}=0$. Proposition \ref{prop4} tells us that $q_i^{(n)}=0$ for all $i,n\in \mathbb{Z}$ with $i\neq n$. We also see that  $p_{i}^{(n)}=r_{j}^{(n)}=0$ for all $i,j,n\in \mathbb{Z}$ with $j\neq 0$.  Again using (\ref{p123}), (\ref{p124}) and (\ref{p125}), we deduce that $\lambda(\frac{m}{2}-n)=(\frac{m}{2}-n)q_{n}^{(n)}$ and thus $\lambda=q_{n}^{(n)}$.
At the same time, it is easy to see that $\rho_3(L_m)=\theta_1(Y_n)+ m \theta_2(Y_n)=0$ for all $m,n\in \mathbb{Z}$, which implies $\rho_3(L_m)=\theta_1(Y_n)= \theta_2(Y_n)=0$.  In addition, it is proved that
$$
\psi(Y_n)=\lambda Y_n+r_0^{(n)}M_0, \ \forall n\in \mathbb{Z}.
$$
Similarly, from Lemma \ref{phipsi} we obtain
\begin{equation*}
f(Y_m,L_n)=(\rho_1(Y_m)+n\rho_2(Y_m)) M_n +[\phi (Y_m), L_n]=\theta_3(L_n)Y_m+[Y_m,\psi(L_n)],
\end{equation*}
and thereby it follows that $\phi(Y_n)=\lambda Y_n+c_0^{(n)}M_0$ and $\rho_1(Y_n)=\rho_2(Y_n)=\theta_3(L_n)=0$ for all $n\in \mathbb{Z}$.
Finally, we have by Lemma \ref{phipsi} that
\begin{equation*}
f(Y_m,L_n)=\rho_3(Y_m)Y_n+[\phi (Y_m), Y_n]=\theta_3(Y_n)Y_m+[Y_m,\psi(Y_n)],
\end{equation*}
which yields that $\rho_3(Y_m)=\theta_3(Y_n)=0$.  The proof is completed.
\end{proof}

\begin{lemma}\label{mform}
Let $\phi$ and $\psi$ be defined by Lemma \ref{phipsi}. Then, there are $w_0^{(n)}, o_0^{(n)}\in \mathbb{C}$ such that
$$
\phi(M_n)=\lambda M_n+w_0^{(n)}M_0,\ \ \psi(M_n)=\lambda M_n+o_0^{(n)}M_0, \ \forall n\in \mathbb{Z},
$$
where $\lambda$ is the same as Lemma \ref{LMLN}. Furthermore,
$$
\rho_i(M_n)=\theta_i(M_n)=0,  i=1,2,3, \ \forall n\in \mathbb{Z}.
$$
\end{lemma}

\begin{proof} For any $n\in \mathbb{Z}$, let
\begin{eqnarray}
\phi(M_n)=\sum_{i\in \mathbb{Z}} u_{i}^{(n)}L_i+\sum_{i\in \mathbb{Z}} v_{i}^{(n)}Y_i+\sum_{i\in \mathbb{Z}} w_{i}^{(n)}M_i, \label{ee233}\\
\psi(M_{n})=\sum_{i\in \mathbb{Z}} d_{i}^{(n)}L_i+\sum_{i\in \mathbb{Z}} l_{i}^{(n)}Y_i+\sum_{i\in \mathbb{Z}} o_{i}^{(n)}M_i, \label{ee333}
\end{eqnarray}
where $u_{i}^{(n)}, v_{i}^{(n)},   w_{i}^{(n)}, d_{i}^{(n)}, l_{i}^{(n)},  o_{i}^{(n)} \in \mathbb{C}$ for every $i\in \mathbb{Z}$. Due to Lemma \ref{LMLN}, by a simple computation, one has
\begin{equation}\label{ad333}
[\phi (L_m), M_n]=-\lambda n M_{m+n},
\end{equation}
and
\begin{eqnarray}
[L_m,\psi(M_n)]&=&\sum_{i\in \mathbb{Z}} (m-i)d_{i}^{(n)}L_{m+i}+
\sum_{i\in \mathbb{Z}}(\frac{m}{2}-i)l_{i}^{(n)}Y_{m+i}-\sum_{i\in \mathbb{Z}} i o_{i}^{(n)}M_{m+i}. \label{ad334}
\end{eqnarray}
By Lemmas \ref{phipsi} and \ref{yform}, we have
$$f(L_m,M_n)=[\phi (L_m), M_n]=(\theta_1(M_n)+ m \theta_2(M_n))M_m+[L_m,\psi(M_n)].$$
This, together with (\ref{ad333}) and (\ref{ad334}), gives that $(m-i)d_{i}^{(n)}=(\frac{m}{2}-i)l_{i}^{(n)}=0$, $\lambda n= n o_{n}^{(n)}$ and
$$\theta_1(M_n)+ m \theta_2(M_n)=0.$$
From this, we deduce that $o_{n}^{(n)}=\lambda$ and $o_{j}^{(n)}=d_{i}^{(n)}=l_{i}^{(n)}=\theta_1(M_n)=\theta_2(M_n)=0$ for all $i,j,n\in \mathbb{Z}$ with $j\neq 0,n$.   Similarly, by considering $f(M_m,L_n)$, we obtain $w_{m}^{(m)}=\lambda$ and $w_{j}^{(m)}=u_{i}^{(m)}=v_{i}^{(m)}=\rho_1(M_m)=\rho_2(M_m)=0$ for all $i,j,m\in \mathbb{Z}$ with $j\neq 0,m$.

Finally,  based on $f(M_m, M_n)= \rho_3(M_m)+[\phi(M_m),M_n]=\theta_3(M_n)+[M_m,\psi(M_n)]$, we deduce that $ \rho_3(M_m)=\theta_3(M_m)=0$ for all $m\in \mathbb{Z}$. The proof is completed.
\end{proof}

For convenience, we define a bilinear map on $\mathcal{SV}(\varepsilon)$ as follows.

\begin{definition} \label{taa}
Let $\Omega=\{\mu_i\in \mathbb{C}|i\in \mathbb{Z}\}$ be a set satisfying $|\Omega|<+\infty$, i.e., $\Omega$ only contains finitely many nonzero numbers. For such an $\Omega$, we define a bilinear map $\chi_\Omega: \mathcal{SV}(\varepsilon)\times \mathcal{SV}(\varepsilon)\rightarrow \mathcal{SV}(\varepsilon)$ given by
\begin{equation}
\chi_\Omega(L_m,L_n)=\sum_{i\in \mathbb{Z}}\mu_i M_{m+n+i}
\end{equation}
for all $m,n\in \mathbb{Z}$ and $\chi_\Omega(x,y)=0$ if either of $x, y$ is contained in $\mathfrak{Y}\cup\mathfrak{M}$.
\end{definition}
 It is easy to verify that $\chi_\Omega$ is a non-inner and non-skewsymmetric biderivation of $ \mathcal{SV}$. In fact, Example \ref{lihai} is just the case of  $f=\chi_\Omega$ with $\Omega=\{\mu_{2016}=2017, \mu_i=0|i\neq 2016\}$.  In contrast, we see that $\chi_\Omega(x,y)$ is symmetric. Thus, if $\chi_\Omega(x,y)$ is also skew-symmetric, then $\chi_\Omega(x,y)=0$.

We now state our main result as follows.

\begin{theorem}\label{maintheo}
$f$ is a biderivation of $\mathcal{SV}(\varepsilon)$ if and only if there are $\lambda\in \mathbb{C}$ and a set $\Omega=\{\mu_i\in \mathbb{C}|i\in \mathbb{Z}\}$ with $|\Omega|<+\infty$ such that
\begin{equation}\label{adform}
f(x,y)=\lambda [x,y]+\chi_\Omega(x,y), \ \ \forall x,y\in \mathcal{SV}(\varepsilon),
\end{equation}
where $\chi_\Omega$ is given by Definition \ref{taa}.
\end{theorem}
\begin{proof}
 The ``if'' direction is easy to verify. We now prove the ``only if'' direction.

 We now assume that  $f$ is a biderivation of $\mathcal{SV}(\varepsilon)$.  Thus, there are $\lambda\in \mathbb{C}$ and a set $\Omega=\{\mu_i\in \mathbb{C}|i\in \mathbb{Z}\}$ with $|\Omega|<+\infty$ such that Lemmas \ref{LMLN}, \ref{yform} and \ref{mform} are  established.  We already know from (\ref{laop1}) that
 $$f(L_m,L_n)=\lambda [L_m,L_n]+\chi_\Omega(L_m,L_n), \ \forall m,n\in \mathbb{Z}.$$
 From our well-known results, it is easy to see that
 \begin{eqnarray*}
 f(L_m,Y_n)=[L_m,\psi(Y_n)]=\lambda [L_m,Y_n]=\lambda [L_m,Y_n]+\chi_\Omega(L_m,Y_n)
 \end{eqnarray*}
 because $\chi_\Omega(L_m,Y_n)=0$. Similarly, we are able to verify one by one that
  \begin{eqnarray*}
 f(x,y)=\lambda [x,y]+\chi_\Omega(x,y)
 \end{eqnarray*}
 for all $x\in \{L_m,Y_m,M_m\}$ and $y\in \{L_n,Y_n,M_n\}$.  For the more general vectors $x,y\in \mathcal{SV}(\varepsilon)$, suppose that
 $$
 x=\sum_{i\in \mathbb{Z}}a_iL_i+\sum_{i\in \varepsilon+\mathbb{Z}}b_iY_i+\sum_{i\in \mathbb{Z}}c_i M_i, \ \  y=\sum_{i\in \mathbb{Z}}k_iL_i+\sum_{i\in \varepsilon+\mathbb{Z}}s_iY_i+\sum_{i\in \mathbb{Z}}t_i M_i.
 $$
 Note that $f(x,y), \lambda[x,y]$ and $\chi_\Omega(x,y)$ are all bilinear with respect to $x,y$, thus, the required results are easy to prove.
\end{proof}

The following corollary is a direct result of this theorem and has already been proved by \cite{WD1} when $\varepsilon=0$.
\begin{corollary}
Every skewsymmetric biderivation of $\mathcal{SV}(\varepsilon)$ is inner.
\end{corollary}

\section{Commutative post-Lie algebra structure}

In this section, we give an application of biderivations. It is well known that the skew-symmetric biderivation can be used to characterize the linear commuting map.  Based on this knowledge, the authors of \cite{WD1} give the forms of the linear commuting map on the Schr\"{o}dinger-Virasoro Lie algebra $\mathcal{SV}(0)$. The same result is valid for the case $\mathcal{SV}(\frac{1}{2})$; we omit the proof here. We shall discuss another application of biderivations: the commutative post-Lie algebra structure on $\mathcal{SV}$. Note that the precondition of this method is that the biderivation is assumed to be non-skewsymmetric.

Post-Lie algebras have been introduced by Valette in connection with the homology of partition
posets and the study of Koszul operads \cite{vela}. As noted in \cite{Burde1}, post-Lie algebras are a natural common generalization of pre-Lie algebras  and LR-algebras in the geometric context of nil-affine actions of Lie groups. Recently, many authors have studied various post-Lie algebras and post-Lie algebra structures  \cite{Burde2,Burde1,Mun,pan,tang2014}. In particular, the authors of \cite{Burde1} studies the commutative post-Lie algebra structure on Lie algebras and, by using the Levi decompositions, proved that any commutative post-Lie algebra structure on a (finite) perfect Lie algebra is trivial.  Note that the Schr\"{o}dinger-Virasoro Lie algebra is an infinite-dimensional perfect Lie algebra. We naturally want to know if this structure is also trivial on a commutative post-Lie algebra. By using our Theorems \ref{maintheo}, we can affirmatively answer this question. Let us recall the following definition of a commutative post-Lie algebra.

\begin{definition}\label{post}
Let $(L, [, ])$ be a complex Lie algebra. A  commutative post-Lie algebra structure on $L$ is a
$\mathbb{C}$-bilinear product $x\cdot y$ on $L$ satisfying the following identities:
\begin{eqnarray}
&& x \cdot y = y\cdot x, \label{post5}\\
&& [x, y] \cdot z =  x \cdot (y \cdot z)-y \cdot (x \cdot z), \label{post6}\\
&& x\cdot [y, z] = [x\cdot y, z] + [y, x \cdot z] \label{post7}
\end{eqnarray}
for all $x, y, z \in V$. We call $(L, [, ], \cdot)$ a commutative post-Lie algebra.
\end{definition}

The following lemma shows the connection between the commutative post-Lie algebra and a biderivation of the Lie algebra.

\begin{lemma}\label{postbide}
Suppose that $(L, [, ], \cdot)$ is a commutative post-Lie algebra. If we define a bilinear map $f : L\times L \rightarrow L$ by $f(x,y)=x\cdot y$ for all $x,y\in L$, then $f$ is a biderivation of $L$.
\end{lemma}

\begin{proof}
For any $x,y,z\in L$, by (\ref{post5}) and (\ref{post7}), we deduce that
\begin{eqnarray*}
 f([x,y],z)=[x,y]\cdot z&=&z \cdot [x,y]=[z\cdot x,y]+[x, z\cdot y]\\
&=&[x\cdot z,y]+[x, y\cdot z]=[f(x,z),y]+[x,f(y,z)],\\
f(x,[y,z])=x\cdot [y,z]&=&[x\cdot y, z]+[y, x\cdot z]=[f(x,y),z]+[y,f(x,z)],
\end{eqnarray*}
which connects (\ref{2der}) and (\ref{1der}), as desired.
\end{proof}

We now give the main result of this section as follows.

\begin{theorem}\label{posttheo}
 Any commutative post-Lie algebra structure on the Schr\"{o}dinger-Virasoro Lie algebra $\mathcal{SV}(\varepsilon)$ is trivial. Namely, $x\cdot y=0$ for all $x,y\in \mathcal{SV}(\varepsilon)$.
\end{theorem}

\begin{proof}
Suppose that $(\mathcal{SV}(\varepsilon), [, ], \cdot)$ is a commutative post-Lie algebra. By Lemma \ref{postbide} and Theorem  \ref{maintheo}, we know that there are $\lambda\in \mathbb{C}$ and a set $\Omega=\{\mu_i\in \mathbb{C}|i\in \mathbb{Z}\}$ with $|\Omega|<+\infty$ such that $x\cdot y=\lambda [x,y]+\chi_\Omega(x,y)$ for all $x,y\in \mathcal{SV}$, where $\chi_\Omega$ is given by Definition \ref{taa}. Because the product $\cdot$ is commutative, we have $\lambda [L_1,L_2]+\chi_\Omega(L_1,L_2)=\lambda [L_2,L_1]+\chi_\Omega(L_2,L_1)$, which implies $\lambda=0$.  By (\ref{post7}), we see that
$$
[L_2,L_1]\cdot L_3=L_2\cdot(L_1\cdot L_3)-L_1\cdot(L_2\cdot L_3).
$$
If there is $\mu_k\in \Omega$ such that $\mu_k\neq 0$, then it is easy to see that the left-hand side of the above equation contains an item $\mu_kM_{6+k}\neq 0$, whereas the right-hand side is equal to zero, which is a contradiction. Thus, we have $\Omega=\{0\}$, i.e., $\mu_i=0$ for any $i\in \mathbb{Z}$. In other words, $\chi_\Omega=0$.
That is, $x\cdot y=0$ for all $x,y\in \mathcal{SV}(\varepsilon)$.
\end{proof}

The proof of Theorem \ref{posttheo} implies that we are able to characterize the commutative post-Lie algebra structure on a Lie algebra $L$ if we know the forms of the biderivation of $L$.  However, the precondition is that the biderivation is assumed to be non-skewsymmetric.

\section{ACKNOWLEDGMENTS}
This work is supported in part by National Natural Science Foundation of China (Grant No. 11171294), Natural Science
Foundation of Heilongjiang Province of China (Grant No. A2015007), and the fund of Heilongjiang Education Committee (Grant No. 12531483).

\end{document}